\documentclass[english]{article}
\usepackage[T1]{fontenc}
\usepackage[utf8]{inputenc}
\usepackage{lmodern}
\usepackage[a4paper]{geometry}
\usepackage{babel}

\usepackage{mathtools}
\usepackage{stmaryrd}
\usepackage{amsmath, amssymb, amsthm}

\usepackage{enumitem} %% Pour l'énumération des sections dans l'intro

\usepackage{nicematrix} %% Pour écrire des matrices avec des labels sur les lignes/colonnes

\usepackage{csquotes}
\usepackage{authblk}

\usepackage[dvipsnames]{xcolor}
\usepackage{soul}%% Pour la commande de highlight \hl

\usepackage{tikz-cd}
\usepackage{biblatex}
\newtheorem{theorem}{Theorem}[section]

\newtheorem{proposition}[theorem]{Proposition}
\newtheorem{corollary}[theorem]{Corollary}
\theoremstyle{definition}
\newtheorem{definition}[theorem]{Definition}
\newtheorem{example}[theorem]{Example}
\newtheorem{lemma}[theorem]{Lemma}
\theoremstyle{remark}
\newtheorem*{remark}{Remark}

\newcommand{\setR}{\mathbb{R}}
\newcommand{\setZ}{\mathbb{Z}}

\newcommand{\setN}{\mathbb{N}}

\newcommand{\HyperbSpace}{\mathbb{H}}

\DeclareMathOperator{\tr}{tr}
\DeclareMathOperator{\ad}{ad}
\DeclareMathOperator{\Ad}{Ad}

\DeclareMathOperator{\End}{End}

\DeclareMathOperator{\Mat}{Mat}

\DeclareMathOperator{\GL}{GL}
\DeclareMathOperator{\SL}{SL}
\DeclareMathOperator{\PGL}{PGL}
\DeclareMathOperator{\PSL}{PSL}

\DeclareMathOperator{\Vect}{Vect}
\DeclareMathOperator{\Stab}{Stab}
\DeclareMathOperator{\id}{id}
\DeclareMathOperator{\Id}{Id}

\DeclareMathOperator{\Grass}{Gr}

\DeclareMathOperator{\im}{im}

\renewcommand{\so}{\operatorname{\mathfrak{so}}} %% Je ne sais pas quelle est la commande \so existante.

\renewcommand{\sl}{\mathfrak{sl}}

\DeclareMathOperator{\lie}{Lie}

\newcommand{\Glie}{\mathfrak g}
\newcommand{\Hlie}{\mathfrak h}
\newcommand{\Klie}{\mathfrak k}

\newcommand{\Plie}{\mathfrak{p}}

\DeclareMathOperator{\opU}{{\mathcal{U}}}

\newcommand{\Orb}{\mathcal{O}}
\DeclareMathOperator{\Flag}{\mathcal{F}}
\newcommand{\Nilp}{\mathcal{N}}

\newcommand{\isom}{{\xrightarrow{\sim}}}

\newcommand{\Killing}{\kappa}
\DeclareMathOperator{\Base}{{\mathcal{B}}}

\renewcommand{\d}{{\operatorname d}} %% \d est par défaut un sous-point (accent)

\newcommand{\sprod}[2]{\left\langle #1 \,,\, #2 \right\rangle}

\newcommand{\argdot}{{\ \cdot \ }}

\renewcommand{\emptyset}{\varnothing}

\title{Nilpotent orbits of $\sl(n,\setR)$ admit no Gibbs state}
\date{\today}

\author{Guillaume Neuttiens}
\affil{Friedrich-Schiller Universität Jena, Germany}
\author{Jérémie Pierard de Maujouy}
\affil{Institut Denis Poisson - Université de Tours, France}

\begin{document}

\maketitle

\begin{abstract}
     Gibbs states are probability distributions defined on Hamiltonian $G$-manifolds. They are parametrized by an open set of the Lie algebra $\Glie$ and arise as natural equilibrium states with regard to the action of $G$ on the system. In this paper, we focus on nilpotent orbits of $\SL(n, \mathbb{R)}$. We show that for $n>2$, those orbits do not admit any Gibbs state.
\end{abstract}

\section{Introduction}

In~\cite{SSDEng}, Souriau introduced a notion of \enquote{statistical equilibrium states} on any symplectic manifold $M$ with a Lie group $G$ of symmetries and a momentum map.
They are probability distributions on $M$ that are parametrized by an infinitesimal symmetry $\beta \in \Glie$ called a \enquote{generalized temperature}.
The interior of the set of generalized temperatures is called the \emph{Gibbs set} of $M$, it is a $\Ad$-invariant convex subset of $\Glie$ that can be empty.

Gibbs sets of symplectic manifolds come with a rich geometry: they have a natural Poisson structure, and as the parameter space for a family of probability measures, they form statistical manifolds.
In particular, they have a natural $G$-invariant Riemannian metric, called the Fisher-Rao metric.
However, they have so far received little attention, and few examples have been studied in detail.
Elementary examples were worked out by Marle~\cite{MarleExamples}, then by a collaboration that includes the two authors~\cite{NilpotentOrbitsGSI23}, as well as~\cite{HessianGas}.

Coadjoint orbits of Lie algebras form in a sense the irreducible blocks of Hamiltonian $G$-manifolds, so their Gibbs sets are of particular interest.
They are the focus of the aforementioned~\cite{NilpotentOrbitsGSI23}, that studied in particular the case of the conjugacy classes of index $2$ nilpotent real matrices of size $3$.
The Gibbs sets correspond to time-like cones in $\setR^3$ equipped with a Lorentzian product and their Fisher-Rao metrics are isometric to a Riemannian direct product $\setR \times \HyperbSpace$, where $\HyperbSpace^2$ is the hyperbolic plane. 
Gibbs sets of coadjoint orbits also belong to a class of exponential families of probabilities that includes many commonly used distributions in statistics and is of interest in information geometry~\cite{TojoYoshino}.

In this paper, we restrict our attention to symmetry under the group $\SL(n, \setR)$.
In this case, the Killing form establishes a correspondance between coadjoint orbits and adjoint orbits, so that coadjoint orbits of linear Lie algebras can be represented as conjugacy classes of matrices.
The nilpotent orbits are of particular interest since they have a discrete, combinatoric classification and form a part of more general coadjoint orbits, according to the Jordan-Chevalley decomposition.

In this paper, we prove that no nonzero nilpotent orbit of $\sl(n, \setR)$ admits generalized temperatures when $n\geqslant 3$ (in contrast with the case $n=2$ which was studied in~\cite{NilpotentOrbitsGSI23}). To this end, we present three different arguments that cover different cases and are each of independent interest. We must mention that the cases handled here are covered by the results of~\cite{NeebGibbsEnsembles}. We think nevertheless that our approach, more concrete, has independent value and may lead in further work to explicit computations of the partition functions.
In Section~\ref{secno:GibbsSets}, we give the definition and the basic properties of the Gibbs set of a Hamiltonian G-manifold.
Section~\ref{secno:AdjointOrbits} is dedicated to generalities on adjoint orbits of semisimple Lie algebras such as their symplectic structure, and their Gibbs sets.
In Section~\ref{secno:slnR} we prove by elimination that nonzero nilpotent adjoint orbits of $\sl(n, \setR)$ admit no generalized temperature when $n\geqslant 3$. Each of the three subsection eliminates a subset of the nilpotent coadjoint orbits:
\begin{enumerate}[label = Section 4.\arabic*:, align=left]% On aimerait énumérer 4.1, 4.2, 4.3
	\item Non-zero nilpotent adjoint orbits stable under the negation map $x\mapsto -x$ never admit generalized temperatures,
	\item Decomposing the partition function into integrals over affine lines shows that nilpotent orbits of index $3$ or higher admit no generalized temperature,
	\item Decomposing the partition function into Lebesgue integrals on matrix spaces shows that nilpotent orbits of index 2 and maximal rank admit no generalized temperature.
\end{enumerate}
General results regarding the structure of the centralizer of a nilpotent matrix, including a determinant formula that is needed in the article, are gathered in Appendix~\ref{secno:centralizer}.

\section{Gibbs sets of Hamiltonian $G$-manifolds}\label{secno:GibbsSets}

In the following we consider a connected symplectic manifold $(M, \omega)$ and  a connected Lie group $G$ with Lie algebra $\mathfrak{g}:=\lie(G)$. Let $\Phi: G\times M\to M$ be a symplectic action of $G$ on $M$. For every $X\in G$, we denote $X^*$ the associated fundamental vector field on $M$ through the action of $\Phi$.
\begin{definition}
   The action $\Phi$ is \emph{Hamiltonian} if it admits a moment map, that is, a  mapping $\mu: M\to \mathfrak{g}^* $ verifying the Hamilton equation
   \begin{equation*}
       \iota_{X^*}\omega=-d\langle \mu, X\rangle
   \end{equation*}
   for all $X\in \mathfrak{g}$.
   
   We say that $\Phi$ is \emph{strongly Hamiltonian} if the moment map is equivariant for the coadjoint action of $G$ on $\mathfrak{g}^*$, i.e. if for all $g\in G, m \in M$, 
   \begin{equation*}
       \mu (\Phi(g,x) ) = \Ad_g^* (\mu(x)).
   \end{equation*}
   In that case, we say that $M$ is a Hamiltonian $G$-manifold.
\end{definition}

 We now fix a $G$-manifold $(M, \omega)$ of dimension $2n$ with momentum map $\mu$ and denote $\lambda_\omega := \omega^n$ the Liouville volume form on $M$.

\begin{definition}[Gibbs states]
    We define $\Omega_M \subseteq \mathfrak{g}$ as the set of all $\beta\in \mathfrak{g}$ such that the positive integral
    \begin{equation}\label{eq:DefPartitionFunction}
    Z(\beta):=\int_M -\exp(-\langle \mu(x), \beta\rangle) \lambda_\omega(x)
    \end{equation}
    is finite.
    To $\beta\in \Omega_M$ is associated a probability distribution on $M$:
    \[
        x \mapsto
        \frac{
        \exp\left( - \sprod{\beta}{\mu(x)} \right) 
        \lambda_\omega (x)
        }
        {Z(\beta)}
    ,\]
    that is called a \emph{Gibbs state}.
\end{definition}%
The interior of $\Omega_M$ is usually called the Gibbs set and its elements are called generalized (or geometrical) temperatures~\cite{MarleGibbsStates, SSDEng, SouriauMecaStat}.

We recall the following important properties, see \cite{SSDEng, MarleExamples} for proofs (usually given for the Gibbs set but they apply to $\Omega_M$)
\begin{proposition}\label{propno:GibbsSetProperties}
    The Gibbs set $\Omega_M$ is convex. If $\beta\in \Omega_M$ then the orbit of $\beta$ in $\mathfrak{g}$ under the adjoint action $\Ad_G$ is contained in $\Omega_M$. In particular, $\Omega_M$ is a union of adjoint orbits.
\end{proposition}

%\begin{definition}
%    The \emph{Gibbs set} $\Omega\subseteq \mathfrak{g}$ is the set of all $\beta\in \mathfrak{g}$ for which there exists a neighbourhood $U$ of $\beta $ and a function $f: M\to \mathbb{R}$ integrable on $M$ with respect to the Liouville measure such that for all $\beta' \in U$, we have 
%    \begin{equation*}
%         -\exp(-\langle \mu(m), \beta\rangle)\leq f(m) \qquad \forall m\in M.
%    \end{equation*}
%\end{definition}%
%By definition, $\Omega$ is open in $\mathfrak{g}$ and for any $\beta\in \Omega$, the expression 
%\begin{equation}%\label{eq:DefPartitionFunction}
%    Z(\beta):=\int_M -\exp(-\langle \mu(m), \beta\rangle) \lambda_\omega(m)
%\end{equation}
%is well-defined. The elements of $\Omega$ are often called \emph{generalised temperatures} in reference to their thermodynamic counterparts. We recall the following important properties, see \cite{SSDEng}, \cite{MarleExamples} for proofs.

\section{Adjoint orbits}\label{secno:AdjointOrbits}

Let $G$ be a connected semisimple Lie group. Then any transitive Hamiltonian $G$-manifold is, up to covering, an orbit of the adjoint action of $G$ on its Lie algebra $\Glie$ (see for example~\cite{KirillovOrbitMethod}). In that sense, adjoint orbits depend only on the Lie considered algebra $\Glie$. Such orbits will be the main focus of this work.

\subsection{Symplectic structure}
\label{secno:AdjointOrbitsSymplectic}

We recall the well-known fact that the coadjoint orbits of a Lie group are Hamiltonian $G$-manifolds for the coadjoint action~\cite{KirillovOrbitMethod}. Let $\xi\in \mathfrak{g}^*$ be nonzero and denote $\mathcal{O}^*_\xi\subseteq \mathfrak{g}^*$ the associated coadjoint orbit. The canonical symplectic structure on $\mathcal{O}^*_\xi$ is the Kirillov-Kostant-Souriau (KKS) form $\omega_{\mathcal{O}^*_\xi}$ defined on the fundamental vector fields by
\begin{equation*}
    \omega_{\mathcal{O}^*_\xi}(X^*, Y^*)(x)=\langle x, [X, Y]\rangle, \qquad x\in \mathcal{O}^*_\xi.
\end{equation*}
It is then easy to check that the coadjoint action $\Ad_G^*$ on $\mathcal{O}^*_\xi$ is strongly Hamiltonian with a moment map $\mu: \mathcal{O}^*_\xi\to \mathfrak{g}^*$ simply given by the embedding of $\mathcal{O}^*_\xi$ in $\mathfrak{g}^*$, that is, $\mu(x)=x$ for all $x\in \mathcal{O}^*_\xi$.

In the case of a semisimple Lie group $G$ (in particular for $G:=SL_n(\mathbb{R}) $ as considered later in this paper), the nondegenerate Killing form induces an isomorphism $\kappa: \mathfrak{g}\to \mathfrak{g}^* $. 
Since $\kappa$ is equivariant under the adjoint action (i.e. $\kappa\circ Ad_G=Ad^*_G \circ \kappa$), it induces a one-to-one correspondance between adjoint orbits of $Ad_G$ in $\mathfrak{g}$ and coadjoint orbits of $Ad_G^*$ in $\mathfrak{g}^*$. 
The symplectic form transported to the orbit $\Orb_x$ of $x\in \Glie$ takes the form
\begin{equation}\label{eqn: KKS_adjoint_orbits}
    \omega_{\Orb_x} (X^*, Y^*) = \kappa(x, [X, Y]). 
\end{equation}
   
Thus, by transporting the KKS structure via $\kappa$, one finds that any nontrivial adjoint orbit is also a Hamiltonian $G$-manifold for the adjoint action $Ad_G$, with a momentum map given by $\kappa$. 

One elementary property of the symplectic structure of adjoint orbits is that it is \enquote{homogeneous} of degree $1$:
\begin{proposition} 
\label{propno:KKSHomogene}
    For $x\in \Glie$, write $\Orb_x$ the adjoint orbit of $x$.
    Let $t\in \setR^*$ be a nonzero real number.
    Then $\Orb_{tx} = t\cdot \Orb_x$ and 
    \[
        (t\Id)^* \omega_{\Orb_{tx}} = t \omega_{\Orb_x}
    .\]
\end{proposition}
\begin{proof}
    The identity $\Orb_{tx} = t\cdot \Orb_x$ is a direct consequence of the linearity of the adjoint action of $G$.
    Let $X, Y, x \in \Glie$. Notice that $t[X, x] = [X, tx]$. Therefore $(t\Id)_*X^* = X^*$ and
    \[
        \omega_{\Orb_tx} (X^*, Y^*) = \Killing(tx, [X,Y]) = t \Killing(x, [X,Y]) = t\omega_{\Orb_x}(X^*, Y^*) 
    ,\]
    which proves the homogeneity of the Kirillov-Kostant-Souriau form. 
\end{proof}

The homogeneity of the symplectic structure implies a homogeneity of the Liouville volume forms.
\begin{corollary}\label{corno:LiouvilleHomogene}
    Let $\Orb_x$ be an adjoint orbit of dimension $2d$. Then for $t\in \setR^*$, the Liouville forms of $\Orb_x$ and $\Orb_{tx}$ are related by the homogeneity relation
    \[
        (t \Id)^* \lambda_{\Orb_{tx}} = t^d \lambda_{\Orb_x}.
    \]
\end{corollary}

As a consequence, the partition function of adjoint orbits satisfies a homogeneity relation.
\begin{corollary}\label{corno:ZHomogene}
    Let $\Orb$ be an adjoint orbit, $\Omega$ be its Gibbs set and $t\in \setR^*$.
    Then the Gibbs set of $t \Orb$ is $t \Omega$ and the partition functions of $\Orb$ and $t\Orb$ are related by:
    \[
        Z_{t\Orb}(\beta) = |t|^{\dim \Orb /2} Z_{\Orb}(t\beta) 
        \quad \text{ for }\beta \in \Omega
    \]
\end{corollary}
\begin{proof}
    Let $\beta\in \Omega$, $t\in \setR^*$, and compute the integral
    \begin{equation*}
        \begin{aligned}
            \int_{t\Orb} e^{-\sprod{\beta}{x}} |\lambda_{t\Orb}|
            &= \int_{\Orb} e^{-\sprod{\beta}{tx}} |t|^{\dim \Orb/2}|\lambda_{\Orb}|\\
            &= |t|^{\dim \Orb/2} \int_{\Orb} e^{-\sprod{t\beta}{x}} |\lambda_{\Orb}|
        .\end{aligned}
    \end{equation*}
\end{proof}
We used in the proof and will be using throughout the article the integration of \emph{densities} on manifolds, see~\cite[Ch. 6]{SemiClassicalAnalysis} for a reference.

\subsection{Gibbs states}

\begin{proposition}\label{propno:noncompactGenTemp}
    Assume that $\Glie$ is a non-compact simple Lie algebra.
    Let $\Orb$ be a nonzero coadjoint orbit, equipped with the Kirillov-Kostant-Souriau symplectic structure. 
    Then $0 \notin \Omega_\Orb$.
\end{proposition}
The proof is inspired from the arguments in~\cite{CompactCoadjointOrbits}

\begin{proof}
    Following the discussion of Section~\ref{secno:AdjointOrbitsSymplectic}, it is equivalent to work with all nonzero \emph{adjoint} orbits, which we will do.
    According to Formula~\eqref{eq:DefPartitionFunction}, if $0\in \Omega_\Orb$% is a generalized temperature
    , then the Liouville measure $\lambda$ on $\Orb$ has a finite volume.
    We will prove that $\lambda(\Orb)$ is infinite.

    Let $x\in \Orb \subset \Glie$ with $x\neq 0$.
    Let $\Klie \oplus \Plie$ be a Cartan decomposition of $\Glie$. Since $\Glie$ is not compact, $\Plie \neq 0$. Furthermore, $\Plie + [\Plie, \Plie]$ is a non-trivial ideal hence equals $\Glie$. In particular, $x$ cannot centralize $\Plie$ because it would imply that $x$ is central and thus zero. Let $h\in \Plie$ such that $[h, x]\neq 0$.

    The endomorphism $\ad_h$ is symmetric for the Euclidean inner product on $\Glie$ associated with the Cartan decomposition, that is $\sprod{\argdot}{\theta (\argdot)}$ with $\sprod{\argdot}{\argdot}$ the Killing form and $\theta$ the Cartan involution. Consequently, $\ad_h$ is diagonalizable over the real numbers. Let us call $X$ the set of its eigenvalues and  decompose $x$ over the eigenspaces of $\ad_h$ : $x = \sum_{\mu \in X} x_\mu$ with $x_\mu\neq 0$ and $[h, x_\mu] = \mu x_\mu$. Since $[h,x]\neq 0$, $x$ has a component with non-zero eigenvalue $z$.
    %Replacing $h$ with $-h$ if needed, let us assume $X$ contains a non-negative eigenvalue $z$.

    We now use $x$ to build an infinite sequence of pairwise disjoint subsets of $\Orb$ with identical nonzero Liouville volume. This will prove that the total Liouville volume is infinite.

    Let $\alpha\in \Glie^*$ be such that $\alpha$ vanishes on $\ker (\ad_h - \mu \id)$ for $\mu \in X \setminus \{z\}$ but $\alpha(x_z)=1$. In particular, 
    \[
        \forall t \in \setR, \;\;
        \alpha(e^{t\ad_h} (x)) = e^{tz}
        .
    \]
    Consider a neighborhood $U$ of $x$ such that $\alpha(U) \subset [1/2,  2]$. Choose $T\in \setR$ such that $M := \exp(z T) > 4$.
    Then for every $n\in \setZ$, $\alpha(e^{nT\ad_h} (U)) \subset [1/2 \times M^n, 2\times M^n]$. In particular, since the intervals $[1/2 \times M^n, 2\times M^n]$ are pairwise disjoint, the subsets $\left\{e^{nT\ad_h} (U)\right\}_{n\in \setZ}$ are pairwise disjoint as well.

    But for any $n\in \setZ$, 
    \[
    \lambda(  e^{nT\ad_h} (U) )
    = \lambda(U) \neq 0
    \]
    since $\lambda$ is $G$-invariant, nonzero and $U$ has a non-empty interior. We thus have constructed an infinite sequence of pairwise disjoint subsets with equal nonzero volume. This achieves the proof.
\end{proof}

\begin{corollary}\label{corno:NoGenTemp}
    Let $\Glie$ be a simple real Lie algebra and $\Orb$ be a nonzero nilpotent orbit. If $\Orb$ contains opposite points (or equivalently $-\Orb = \Orb$), $\Orb$ admits no Gibbs state.    
\end{corollary}
\begin{proof}
    If $\Glie$ admits nonzero nilpotent elements, $\Glie$ is noncompact.
    If $\Orb$ is stable under under negation, Corollary~\ref{corno:ZHomogene} implies that $\Omega_\Orb$ is stable under negation as well. Thus if $\Omega_\Orb$ is non-empty, it necessarily contains $0$ by convexity (Proposition~\ref{propno:GibbsSetProperties}), which contradicts Proposition~\ref{propno:noncompactGenTemp}.
\end{proof}

\section{Nilpotent adjoint orbits of $\sl(n, \setR)$}
\label{secno:slnR}

For any integer $n$, the group $\SL(n, \setR)$ is connected. In consequence, the nilpotent adjoint orbits of $\sl(n, \setR)$ are exactly the conjugacy classes under $\SL(n, \setR)$ of nilpotent matrices.

Since we will not handle abstract linear forms on $\sl(n, \setR)$, we shall use the notation $\sprod{\argdot}{\argdot}$ for the Killing form.

\subsection{Negation map}
\label{secno:NegationMap}

According to Corollary~\ref{corno:NoGenTemp}, nonzero nilpotent orbits that are stable under negation cannot admit Gibbs states. In this section, we study the question of which nilpotent orbits are stable under the negation map $x\mapsto -x$.

The conjugacy classes under $\GL(n, \setR)$ of nilpotent matrices are characterized by the set of their Jordan blocks in the Jordan normal form: the number of Jordan blocks of each size totally characterizes a nilpotent matrix up to conjugacy. Since the sizes of the Jordan blocks sum up to $n$, they form a partition of $n$.
In other words: conjugacy classes of nilpotent matrices under $\GL(n, \setR)$ are in bijection with partitions of $n$.

The conjugation action of $\GL(n, \setR)$ factors through the projective linear group $\PGL(n, \setR)$. 
In this quotient by the center, the subgroup $\SL(n, \setR)\subset \GL(n, \setR)$ projects to $\PSL(n, \setR) \subset \PGL(n, \setR)$. 
Since $\PSL(n, \setR)$ is a subgroup of index either $1$ or $2$, each $\PGL(n, \setR)$-orbit either is a $\PSL(n, \setR)$-orbit or decomposes into two $\PSL(n, \setR)$-orbits, which are exchanged under conjugation by a matrix with negative determinant.

\begin{lemma}
Let $x\in \sl(n, \setR)$.
\begin{enumerate}
    \item If all matrices in $Z_{\GL(n, \setR)}(x)$ have positive determinant, then $\Ad_{\GL(n, \setR)}(x)$ splits into two $\SL(n, \setR)$ orbits,
    \item If $x$ commutes with a matrix of negative determinant, then $\Ad_{\SL(n, \setR)}(x) = \Ad_{\GL(n, \setR)}(x)$.
\end{enumerate}
\end{lemma}

\begin{proof}
Let us write $\GL_n^-(\setR)\subset \GL(n, \setR)$ (resp. $\GL_n^+(\setR)$) for the subset of matrices with negative (resp. positive) determinant. Note first that $\Ad_{\GL_n^+(\setR)}(x) = \Ad_{\SL(n, \setR)}(x)$. 
If $x$ commutes with a matrix of negative determinant, that is, if there exists $M\in \GL_n^-(\setR)$ such that $MxM^{-1} = x$ then 
\[
    \Ad_{\GL_n^-(\setR)}(x)
    = \Ad_{\GL_n^+(\setR)M}(x)
    = \Ad_{\GL_n^+(\setR)}(x)
    = \Ad_{\SL(n, \setR)}(x)
.\]
Therefore $\Ad_{\GL(n, \setR)}(x) = \Ad_{\SL(n, \setR)}(x)$ and $\SL(n, \setR)$ acts transitively on the $\GL(n, \setR)$-conjugacy class of $x$.

Conversely, if $\Ad_{\GL(n, \setR)}(x) = \Ad_{\SL(n, \setR)}(x)$, then $\Ad_{\GL_n^-(\setR)}(x) = \Ad_{\GL_n^+(\setR)}(x)$ which implies the existence of two matrices $M_-, M_+$, respectively in $\GL_n^\pm (\setR)$, such that $M_- x M_-^{-1} = M_+ x M_+^{-1}$. In consequence, $M_+^{-1} M_-$ commutes with $x$ and has a negative determinant.
We have proven that $\SL(n, \setR)$ acts transitively on the $\GL(n, \setR)$-conjugacy class if and only if $x$ commutes with a matrix of negative determinant.
\end{proof}

We are interested in whether nilpotent orbits of $\SL(n, \setR)$ contain opposite points, that is, whether given a nilpotent $x\in \sl(n, \setR)$ it holds $-x\in \Ad_{\SL(n, \setR)}(x)$.

\begin{proposition}
    Let $x\in \sl(n, \setR)$ be a nilpotent element.
    Then $x$ and $-x$ belong to the same $\SL(n, \setR)$-orbit if and only one of the following holds:
    \begin{itemize}
        \item $x$ has a Jordan blocks of odd dimension, or
        \item $n \not\equiv 2 \mod 4$.    
    \end{itemize} 
    In that case, the adjoint orbit of $x$ admits no Gibbs state.
\end{proposition}

\begin{proof}
    First of all, since $x$ and $-x$ have the same Jordan form, they are conjugate under $\GL(n, \setR)$.
    
    %We consider the case when $n$ is odd : $-I_n$ commutes with $x$ and has determinant $-1$ thus the $\GL(n, \setR)$-orbit is a $\SL(n, \setR)$-orbit and $x$ and $-x$ are $\SL(n, \setR)$-conjugate.

    %We now assume that $n$ is even.
    Let $(m_i)_{1\leqslant i \leqslant n}$ be the multiplicities of the Jordan block of dimension $i$ in $x$: $\sum_{i} m_i i = n$. According to Proposition~\ref{propno:DeterminantCentralizer}, $Z_{\GL(n, \setR)}(x)$ contains elements of negative determinant if and only if there exists an odd dimension $i$ with a nonzero multiplicity $m_i$. In this case, $\SL(n, \setR)$ acts transitively on the $\GL(n, \setR)$-orbit hence conjugates $x$ and $-x$.

    We now assume that all Jordan blocks have even dimension; in particular the dimension $n$ is even. The conjugacy class of $x$ decompose into two $\SL(n, \setR)$-orbits. 
    Let $x'$ be a Jordan form of $x$: the only nonzero entries of $x'$ are right above the diagonal. Note that $x'$ and $-x'$ belong to the same $\SL(n, \setR)$-orbit if and only if $x$ and $-x$ belong to the same $\SL(n, \setR)$-orbit. Define the following matrix
    \[
        S:= \begin{pNiceMatrix}
                1 & & & & \\
                 & -1& & & \\
                 & & 1& & \\
                 & & & -1& \\
                 & & & & \ddots
        \end{pNiceMatrix}
        \in \GL(n, \setR)
    .\]
    Then $\det(S) = (-1)^{n/2}$ and $Sx'S^{-1} = -x'$. 
    Therefore:
    \begin{itemize}
        \item If $n \equiv 0 \mod 4$, $S\in \SL(n, \setR)$ and $x'$ is $\SL(n, \setR)$-conjugate to $-x'$, and the same holds for $x$ and $-x$.
        \item If $n \equiv 2 \mod 4$, $\det(S)=-1$ and $x'$ and $-x'$ belong to different $\SL(n, \setR)$-orbits, and the same holds for $x$ and $-x$.
    \end{itemize}

    The absence of Gibbs states in a consequence of Corollary~\ref{corno:NoGenTemp}.
\end{proof}

\subsection{Partition into affine lines}
\label{secno:AffineLines}

Consider the set of nilpotent matrices of size $n$ with nilpotency index $r+1$ and assume that $r+1 >2$. There is an action of $\setR$ as follows:
\begin{equation}\label{eqno:ActionRonNilpotent}
    \phi : (t, x) \mapsto x + t x^r
.\end{equation}
The orbits of $\phi$ are affine lines, with an affine parametrization.

On the other hand, the function $\xi\in \sl(n, \setR) \mapsto e^{-\sprod\beta\xi}$ is not (Lebesgue-)integrable on any affine line in $\sl(n, \setR)$, no matter $\beta\in \sl(n, \setR)^*$.
In this section, we use this observation to show that the nilpotent orbits of index $r+1 >2$ have no Gibbs state. To this end, we need to show that:
\begin{itemize}
    \item the action~\eqref{eqno:ActionRonNilpotent} preserves nilpotent orbits,
    \item the induced measure on the affine lines corresponds to Lebesgue measure.
\end{itemize}

Before proving these, we need to relate the integrability on a nilpotent orbit to the integral on each affine line. This is the object of Fubini-Tonelli theorem.
Since it is not easy to find a reference with a suitable form of the theorem, we provide a short proof.

\begin{theorem}[Fubini-Tonelli theorem]\label{thmno:Fubini}
    Let $E\xrightarrow{\pi} B$ be a fiber bundle, $\omega_E$ a volume form on $E$, $\omega_B$ a volume form on $B$ and $f: E \to \setR_+$ a smooth function.

    The isomorphism of line bundles 
    \[
        \Lambda^{\dim E} T^*E \simeq \Lambda^{\dim B} \pi^*{T^* B} \otimes \Lambda^{\dim E - \dim B} (\ker \d \pi )^*
    \]
    allows defining a divided fiber volume form $\frac{\omega_E}{\pi^*\omega_B}$ on $E$.

    Writing fibers as $E_b := \pi^{-1}(b)$, there is an equality of positive integrals
    \[
        \int_{E} f |\omega_E| 
        = \int_B \; \left[ 
                \int_{E_b} f \left| \frac{\omega_E}{\pi^*\omega_B} \right|    
            \right] \;|\omega_B(b)|
        \in [0, +\infty]
    \]
    where $| \argdot |$ denote the absolute value of a volume form, which is a positive density.
\end{theorem}

\begin{proof}
    By the standard argument of a partition of unity, it is sufficient to prove the theorem over each member of an open cover of $E$. 
    Using a trivialization cover reduces the problem to the case of a product manifold $E = F\times B$, where the standard Fubini-Tonelli theorem applies. 
    
    Let $\omega_F$ be a volume form on $F$. Then there exists $u : F \times B \to \setR^*$ such that $\omega_E= u \,\omega_F \wedge \pi^*\omega_B $ (with an implicit pullback of $\omega_F$ to $F\times B$), namely $\frac{\omega_E}{\pi^*\omega_B} = u \,\omega_F$.
    
    Writing $\mu_F$ and $\mu_B$ the Borel measures respectively associated with $|\omega_F|$ and $|\omega_B|$, it can be checked locally that $\mu_F\otimes \mu_B$ is the Borel measure associated with $|\omega_F\wedge \omega_B|$, so that the Fubini-Tonelli theorem applied to the positive function $f|u|$ allows concluding.
\end{proof}

\begin{corollary}\label{corno:Fubini}
    Under the same hypotheses, if $\int_{E} f |\omega_E| < \infty$, 
    then $\int_{E_b} f \left| \frac{\omega_E}{\pi^*\omega_B} \right| < \infty$ for almost every $b\in B$.
\end{corollary}

We now establish the required properties of $\phi$.
\begin{lemma}
    \begin{enumerate}
        \item The action $\phi$ commutes with the adjoint action.
        \item The nilpotent adjoint orbits of index $r+1$ are stable under the action $\phi$.
    \end{enumerate}
\end{lemma}
\begin{proof}
    For every $x\in \sl(n,\setR)$, $t\in \setR$ and $g\in \SL(n, \setR)$, 
    \[
        \Ad_g(x + t x^r) = \Ad_g (x) + t (\Ad_g(x))^r
    ,\]
    which proves the first point.

    The $\operatorname{GL}(n, \setR)$-conjugacy classes of nilpotent matrices are characterized by the set of Jordan blocks. We assume without loss of generality that $x$ is in a Jordan form, and thus $\phi_t(x)$ can be computed blockwise.
    The maximal size of the Jordan blocks of $x$ is $r+1$; blocks of size $r$ or lower are left invariant under $\phi$. A block of size $r+1$ is mapped as follows under $\phi$:
    \[
        \phi_t\left( 
            \begin{pmatrix}
                0 & 1 & 0 & \dots & 0 \\
                0 & 0 & \ddots & 0 & 0\\
                0 & \dots &  & \ddots & 0\\
                0 & \dots &    &  0 & 1\\
                0 & \dots &    &    & 0
            \end{pmatrix}
            \right)
    = 
            \begin{pmatrix}
                0 & 1 & 0 & \dots & t \\
                0 & 0 & \ddots & 0 & 0\\
                0 & \dots &  & \ddots & 0\\
                0 & \dots &    &  0 & 1\\
                0 & \dots &    &    & 0
            \end{pmatrix}
    .\]
    The image under $\phi_t$ of the Jordan block of size $r+1$ is still a rank $r$ nilpotent matrix,
    %with a cyclic vector,
    therefore belongs to the same $\operatorname{GL}(r, \setR)$-conjugacy class. 
    Since the dependency on $t$ is continuous and the identity component of $\operatorname{GL}(r, \setR)$ is $\operatorname{SL}(r, \setR)$, $\phi$ preserves the $\operatorname{SL}(r, \setR)$-conjugacy class of each Jordan block.
    Consequently, $\phi$ preserves the $\operatorname{SL}(n, \setR)$-conjugacy class of $x$.
\end{proof}    
\begin{proposition}\label{propno:LiouvillePhiInvariant}
    The Liouville form of nilpotent adjoint orbits is invariant under $\phi$.
\end{proposition}  
\begin{proof}
    Let $x$ be a nilpotent matrix of index $r+1$ with an adjoint orbit $\Orb$ of dimension $2d$ and denote $\omega$ the KKS form on $\Orb$ and $\lambda=\omega^d$ the Liouville form.
    Since $\phi$ commutes with the adjoint action, $\phi$ leaves any fundamental vector field invariant. Let $X, Y \in \sl(n, \setR)$. Then for $t\in \setR$,
    \[
        \phi_t^*(\omega)|_x(X^*,Y^*)
            = \omega|_{x+t x^r}(X^*, Y^*)
            = \omega|_x(X^*, Y^*) + t\Killing(x^r, [X,Y])
    ,\]
    where the second equality arises form the expression (\ref{eqn: KKS_adjoint_orbits}) of the Liouville form on adjoint orbits.
    Let us define on the adjoint orbit the $2$-form $\omega_r$ such that
    \[
        \omega_r|_x(X^*, Y^*) = \kappa(x^r, [X,Y])
    .\]
%        Since for every $t$, $(x+tx^r)^r = x^r$, $\omega_r$ is invariant under $\phi$.

    Since $\phi$ commutes with the adjoint action, it acts on the line of $\Ad$-invariant volume forms. This is a linear action, and there exists $a\in \setR$ such that for every $t\in \setR$, 
    \[
        \phi_t^*\lambda = e^{at} \lambda
    .\]
    However, a direct computation gives
    \[
        \phi_t^* \lambda
        = \phi_t^*\omega^{d}
        = \sum_{i=0}^d \binom{d}{i} t^i \omega^{d-i} \wedge \omega_r^i
    .\]
    The only possibility is $a=0$, namely $\phi$ acts trivially on the invariant volume forms.
\end{proof}

The decomposition of nilpotent adjoint orbits into $\phi$-orbits takes the form of a principal bundle:
\begin{proposition}
    Let $\Orb$ be a nilpotent adjoint orbit in $\sl(n, \setR)$ with nilpotency index $r+1 > 2$. Then
    \begin{enumerate}
        \item the action $\phi$ of $\setR$ on $\Orb$ is free, 
        \item the quotient map $\Orb \xrightarrow{\pi} \Orb / \setR$ defines a $\setR$-principal bundle map,
        \item for any $G$-invariant volume form $\omega$ on $\Orb/\setR$, the vertical volume form $\frac{\lambda}{\pi^*\omega}$ is $\phi$-invariant.
    \end{enumerate}
\end{proposition}

\begin{proof}
    Assertion 1. is immediate from Equation~\ref{eqno:ActionRonNilpotent}, since $x^r \neq 0$.
    For 2, we use the standard result according to which a smooth, free and proper action of a Lie group has a smooth quotient and the quotient map defines a principal bundle~\cite{LeeManifolds}. We thus need to prove that $\phi$ defines a proper action. Let $K \subset \Orb$ be a compact subset. The subset $K^{(r)}:=\{x^r\}_{x\in K} \subset \sl(n, \setR)$ is compact and does not contain $0$. There exists therefore a linear form $f \in \sl(n, \setR)^*$ and a closed interval $[\alpha, \beta]\subset \setR^*_+$ such that $f(K^{(r)}) \subset [\alpha,\beta]$. We have $f(\phi_t (K)) \subset f(K) + t[a,b]$.
    Let $K' \subset \Orb$ be another compact subset and let $[a,b]$ (resp. $[a',b']$) be an interval containing $f(K)$ (resp. $f(K')$). Then for any $t\in \setR$, 
    \[\begin{aligned}
        \phi_t (K) \cap  K' \neq \emptyset
        &\implies 
        f(\phi_t(K)) \cap  f(K') \neq \emptyset
        \\
        &\implies 
        \big([a, b] + t[\alpha, \beta] \big) \cap [a', b'] \neq \emptyset
        \\
        &\implies 
        t[\alpha, \beta] \cap [a' -b, b'-a] \neq \emptyset
        \\
        &\implies
        t \in \left[
            \frac{a'-b}{\beta}, 
            \frac{b'-a}{\alpha}
        \right] 
    .\end{aligned}\]
    This proves that the action $\phi$ is proper and that the corresponding quotient defines a principal bundle.

    Finally, if $\omega$ is $G$-invariant then $\pi^*\omega$ is $G$-invariant and since $\lambda$ is $G$-invariant (Proposition~\ref{propno:LiouvillePhiInvariant}, their quotient is $G$-invariant as well.
\end{proof}

We now have everything to prove the divergence of the partition function on nilpotent orbits of index of $3$ or higher.
\begin{theorem}
    Nilpotent orbits of $\sl(n, \setR)$ with nilpotency index strictly greater than $2$ admit no Gibbs state.
\end{theorem}
\begin{proof}
    Let $\Orb$ be a non-minimal nilpotent orbit with a nilpotency index $r+1>2$ and let $\beta\in \sl(n, \setR)$.
    Let $\left| \omega \right|$ be any non-vanishing density on the orbit manifold $\Orb/\setR$. 
    According to Corollary~\ref{corno:Fubini}, if $\xi \mapsto e^{-\sprod \beta \xi} |\lambda|$ has a finite integral then on almost every orbit of the action $\phi$, $\xi \mapsto e^{-\sprod \beta \xi}  \frac{|\lambda|}{\pi^*|\omega|}$ has a finite integral.

    Let $x\in \Orb$. Since the vertical density $\frac{|\lambda|}{\pi^*|\omega|}$ is invariant under $\phi$, the integral of $\xi \mapsto e^{-\sprod \beta \xi} \frac{|\lambda|}{\pi^*|\omega|}$ on $\Orb$ is proportional to the integral
    \[
        \int_\setR e^{-\sprod{\beta}{x+tx^r}} \d t = e^{-\sprod{\beta}{x}} \int_\setR e^{-t\sprod{\beta}{x^r}} \d t 
    \]
    which is never finite.
\end{proof}

\subsection{Nilpotent orbits of index $2$ and rank $n/2$}
\label{secno:Grassmannian}

In this section, we handle the remaining nilpotent orbits of $\sl(n, \setR)$. The remaining case is when all Jordan blocks have dimension $2$. Namely, $n=2m$ and we are looking at the orbits of the blockwise matrices
\[
	\pm
	\begin{pmatrix}
		0 & I_m\\
		0 & 0
	\end{pmatrix}
.\]

Let us fix $m\in \setN_+$. We know from Section~\ref{secno:NegationMap} that the orbit cannot have Gibbs states if $m$ is even, but we will not need to assume $m$ to be odd.
Here again, we approach the calculation of the partition function using a fibration.

Let $\Nilp$ be the set of nilpotent matrices of index $2$, with all Jordan blocks of size $2$. We know that $\Nilp$ decomposes into two $\SL(2m, \setR)$-orbits according to Section~\ref{secno:NegationMap}. Consider the following map to the $m^2$-dimensional Grassmannian $\Grass(m, 2m)$:
\[
	x \in \Nilp \xrightarrow{\pi} \ker x \in \Grass(m, 2m)
.\]
It is equivariant under the action of $\SL(2m, \setR)$, because $\ker (gxg^{-1}) = g (\ker x)$.
Given a $m$-plane $E\in \Grass(m,2m)$, the action of $\{g \in \SL(2m, \setR) \, | \, g(E) = E\}$ is transitive over the fiber $\pi^{-1}(E)$.

Let us consider the $m$-plane $E_0 = \setR^m \times 0 \subset \setR^{2m}$. The fiber $\pi^{-1}(E_0)$ takes the form
\[
	\pi^{-1}(\setR^m \times 0) = 
		\left\{
		\begin{pmatrix}
			0 & M\\
			0 & 0
		\end{pmatrix}
		\, | \,
		M \in \GL(m, \setR)
		\right\}
.\]
The stabilizer of $E_0$ in $\SL(2m)$ is the subgroup
\[
	H:=
	\Stab(\setR^m \times 0) 
	= \left\{
		\begin{pmatrix}
			A & C\\
			0 & B
		\end{pmatrix}, 
		A, B, C \in \Mat(m, \setR)
		\, | \,
		\det(A)\det(B) = 1
	\right\}
.\]
The action of $H$ on $\pi^{-1}(E_0)$ is as follows:
\begin{equation}\label{eqno:ActionHfiber}
	\Ad_{\begin{pmatrix}
				A & C\\
				0 & B
			\end{pmatrix}}
	\left(
		\begin{pmatrix}
				0 & M\\
				0 & 0
			\end{pmatrix}
	\right)
	=
		\begin{pmatrix}
				0 & AMB^{-1}\\
				0 & 0
			\end{pmatrix}
.\end{equation}
In particular, this gives the decomposition of the fiber $\pi^{-1}(E_0)$ into its intersections with the two nilpotent orbits of $\Nilp$:
\begin{equation*}
	F_{\pm} = 
		\left\{
		\begin{pmatrix}
			0 & M\\
			0 & 0
		\end{pmatrix}
		, M \in \GL(m, \setR)
		\, | \,
		 \pm \det(M) > 0
		\right\}
.
\end{equation*}
In particular, the fibers are readily parametrized by the connected group $\GL^+(m) = \{M\in \GL(m, \setR) \, |\, \det(M) > 0\}$.

In order to compute the partial integral of $e^{-\beta}$ on the fibers $F_{\pm}$, we need to identify the fiber volume form. This requires studying the behavior of $H$ on the tangent space $T_{E_0} \Grass(m,2m)$.

\begin{lemma}\label{lmno:TangentGrass}
	The tangent space $T_{E_0} \Grass(m,2m)$ is a representation of $H$ that is isomorphic to $\Mat(m, \setR)$ with the action
	\[
		\begin{pmatrix}
						A & C\\
						0 & B
					\end{pmatrix}
		\cdot
		X
			=
		B X A^{-1}
.	\]
\end{lemma}

\begin{proof}
	It is a standard fact of the theory of homogeneous spaces that $T_{E_0} \Grass(m,2m)$ is a representation of $H$ that is isomorphic to $\sl(2m, \setR)/\Hlie$, with $\Hlie$ the Lie algebra of $H$~\cite{MichorTopicsDG}.
	$\Hlie$ can be described as
	\[
		\left\{
			\begin{pmatrix}
				a & c \\
				0 & b
			\end{pmatrix}
			, a, b, c \in \Mat(m, \setR)
			\, | \,
			\tr(a) + \tr(b) = 0		
		\right\}
	\]
	so that the quotient space $\sl(2m, \setR)/\Hlie$ has an isomorphic description
	\[
		\left[
			\begin{pmatrix}
				a & c \\
				d & b
			\end{pmatrix}
		\right]	
		\in \sl(2m, \setR)/\Hlie
		\mapsto 
		d \in \Mat(m, \setR)
	.\]
	The action of $H$ is readily obtained from the action on $\sl(2m, \setR)$, which we write partially:
\begin{equation*}
	\Ad_{\begin{pmatrix}
				A & C\\
				0 & B
			\end{pmatrix}}
	\left(
		\begin{pmatrix}
				* & *\\
				d & *
			\end{pmatrix}
	\right)
	=
		\begin{pmatrix}
				* & *\\
				BdA^{-1} & *
			\end{pmatrix}
.\end{equation*}
\end{proof}
We can then compute the determinant of the action of $H$ on $T_{E_0} \Grass(m,2m)$:
\begin{lemma}\label{lmno:detTEGrass}
\begin{enumerate}
	\item Let $A, B \in \GL(m, \setR)$. Then 
	\[
		\det\nolimits_{\setR^m \otimes \left(\setR^m \right)^*} (A\otimes B^T)
		= \det(A)^m \det(B)^m
	.\]
	\item The element
		$\begin{pmatrix}
			A & C\\
			0 & B
		\end{pmatrix} \in H$
		acts on $T_E \Grass(m,2m)$ with determinant $\det(B)^{2m}$.
\end{enumerate}
\end{lemma}
\begin{proof}
\begin{enumerate}
	\item $A\otimes B^T = (A \otimes \id_{\left(\setR^m \right)^*}) \circ (\id_{\setR^m} \otimes B^T)$. Furthermore, $\det(B^T) = \det(B)$, thus it is sufficient to prove the identity for $B=\id_{\left(\setR^m \right)^*}$. 
	Then the action of $A \otimes \id_{\left(\setR^m \right)^*}$ on $\setR^m \otimes \left(\setR^m \right)^*$ is isomorphic to the standard action of $A$ on a direct sum of $m$ copies of $\setR^m$. Therefore 
	\[
		\det(A \otimes \id_{\left(\setR^m \right)^*})
		= 
		\det(A)^m
	,\]
	which is sufficient to conclude.
	\item Let $A, B \in \GL(m, \setR)$.
    Under the natural isomorphism 
    $\Mat(m, \setR) \simeq \setR^{m}\otimes \left( \setR^m \right)^*$,
    the map $X \in \Mat(m, \setR) \mapsto B X  A^{-1}$ is linearly isomorphic to $B\otimes (A^{-1})^T \in \End\left( \setR^m \otimes \left(\setR^m \right)^* \right)$. Assuming $\det(A)\det(B)=1$, the previous point leads to a determinant of
	\[
		\det(B\otimes (A^{-1})^T)
		= \det(B)^m \det(A^{-1})^m
		= \det(B)^{2m}
	\]
	as asserted.
\end{enumerate}
\end{proof}

Let $\omega_0 \in \Lambda^{m^2} T_{E_0}^* \Grass(m,2m)$ be a nonzero volume element. It is possible to define a horizontal form $\pi^*\omega_0$ along $\pi^{-1}({E_0})$.
\begin{lemma}
	Let $h=\begin{pmatrix}
								A & C\\
								0 & B
							\end{pmatrix} \in H$.
	Then 
	\begin{equation}\label{eqno:QuasiInvarianceVerticalForm}
		\Ad_h^*\left(\frac{\lambda}{\pi^*\omega_0}\right)
		= \det(B)^{-2m} \left(\frac{\lambda}{\pi^*\omega_0}\right)
	.\end{equation}
\end{lemma}
\begin{proof}
	The Liouville form is invariant under the adjoint action.
	The action of $\Ad_h$ on $\pi^*\omega_0$ can be computed using Lemma~\ref{lmno:detTEGrass}:
	\[
		\Ad_h^*(\pi^*\omega_0)
		= \pi^*(\Ad_h^*|_E \, \omega_0)
		= \pi^*(\det(B)^{2m} \omega_0)
		= \det(B)^{2m} \pi^*\omega_0
	.\]
    The first equality holds because $\pi$ is $\SL(2m, \setR)$-equivariant.
	This is sufficient to obtain~\eqref{eqno:QuasiInvarianceVerticalForm}, by division.
\end{proof}

Since $H$ acts transitively on $F_\pm$, Equation~\eqref{eqno:QuasiInvarianceVerticalForm} characterizes the vertical volume form on each orbit, up to a constant.

\begin{proposition}\label{propno:LebesgueMeasure}
	The diffeomorphisms
	\begin{equation}\label{eqno:Nilpotent->Grass}
		\begin{pmatrix}
			0 & M\\
			0 & 0
		\end{pmatrix}
		\in F_{\pm}
		\mapsto 
		M \in \GL^\pm (m)
	\end{equation}
	map the fiber volume form
	$\frac{\lambda}{\pi^*\omega_0}$
	to restrictions of Lebesgue measures of the space $\Mat(m, \setR)$ of all real square matrices of size $m$.
\end{proposition}
\begin{proof}
	We discuss the case of $F_+$; the situation of $F_-$ is similar. Let us call $f$ the map~\eqref{eqno:Nilpotent->Grass}. We will show that the image volume form under $f$ has the same quasi-invariance property as the Lebesgue measure, and conclude that they are proportional.
    
	Let $h=\begin{pmatrix}
								A & C\\
								0 & B
							\end{pmatrix} \in H$.
	According to Equation~\eqref{eqno:ActionHfiber}, 
	\[
		f(\Ad_h y) = A f(y) B^{-1}
	.	\]
	Let $\eta := f_*\left(\frac{\lambda}{\pi^*\omega_0}\right)$ be the image in $\GL^+(m)$ of the vertical volume form on $F_+$, with $\lambda$ the Liouville form on $\Nilp$.
	Equation~\eqref{eqno:QuasiInvarianceVerticalForm} implies that
	\[
	f_*\left(
		\Ad_h^* \frac{\lambda}{\pi^*\omega_0}
		\right) 
		= \det(B)^{-2m} \eta
	.\]
	We write $L_A$ for the left action of $A$ on $\GL(m, \setR)$ and $R_B$ for the right action of $B$.
	Combining the two above equations yields
	\begin{equation}\label{eqno:EquivVolumeGLm}
		\forall A, B \in \GL(m, \setR),\,
		 \det(A)\det(B) = 1 
		\implies 
		L_A^*R_{B^{-1}}^* \eta = \det(B)^{-2m} \eta
	.\end{equation}
	
	We need to investigate the volume forms on $\GL^+(m)$ that satisfy Equation~\eqref{eqno:EquivVolumeGLm}. Let $A\in \GL^+(m)$ and define $z=\det(A)^{-1/2m}$. Then $\det(z A) \det(z I_m) = 1$ and 
	\[
		\forall M \in \GL(m, \setR), \,
		zA M (z I_m)^{-1} = A M
	.\]
	This implies the following quasi-invariance of $\eta$ under left translations:
	\[
		\forall A \in \GL^+(m), 
		\,
		L_A^* \eta
		=  \det(A)^{m} \eta
	\]
%	and there is a similar quasi-invariance under right translations:
%	\[
%		\forall B \in \GL^+(m), 
%		\,
%		R_B^* \eta
%		=  \det(B)^{m} \eta
%	.\]
%%	
%	Finally, according to Lemma~\ref{lmno:detTEGrass},
%	\[
%		\det(A\otimes B^T) = \det(A)^m \det(B)^m
%	\]
%	is the determinant of the linear map $M\in \Mat(m)\mapsto AMB$. Therefore, the quasi-invariance of $\eta$ under right and left translations coincides with that of (the restriction of) Lebesgue measures on $\GL(m, \setR)$. Since this action of $\GL(m, \setR)\times \GL(m, \setR)$ [...]
    which is the same quasi-invariance than the restriction of the Lebesgue measure to $\GL(m, \setR)$, according to the determinant formulas from Lemma~\ref{lmno:detTEGrass}.
    Since the action of $\GL(m, \setR)$ by left translations
    is transitive, this property characterizes the Lebesgue measures, and proves that $\eta$ coincides with the restriction to $\GL(m, \setR)$ of a Lebesgue measure on $\Mat(m, \setR)$.
\end{proof}
    Note that $\GL^\pm(m, \setR)$ are two open cones in $\Mat(m, \setR)$.
    Having identified the vertical measure, it is now easy to recognize a non-integrable behavior in the integrand. 
\begin{lemma}\label{lmno:NegativetrAM}
	If $m>1$ then for any $A\in \Mat(m, \setR) \setminus\{0\}$ there exists $M\in \GL^+(m)$ such that $\tr(AM)<0$.
\end{lemma}
\begin{proof}
	Let $A\in \Mat(m, \setR) \setminus\{0\}$, let $d\geqslant 1$ be its rank. There exist $P,Q\in \GL(m, \setR)$ such that
	\[
		P A Q^{-1} = 
		\begin{pmatrix}
		I_d & 0 \\
		0 & 0
		\end{pmatrix}
	.\]
	Define the matrices
	\[
		D_\pm =
		\begin{pmatrix}
			-m &  &  &  &  &  \\
			   &\pm1&  &  &  &  \\
			   &  & 1&  &  &  \\
			   &  &  & 1&  &  \\
			   &  &  &  & 1&  \\
			   &  &  &  &  & \ddots 
		\end{pmatrix}
		\in \GL^{\mp}(m)
	.\]
	Then 
	\[
		\tr(A (Q^{-1}D_\pm P))
		= \tr(PAQ^{-1} D_\pm)
		\leqslant -2 \pm 1 < 0
	\]
	and depending on the sign of $\det(P)\det(Q^{-1})$, either $Q^{-1} D_+ P$ or $Q^{-1} D_- P$ has the wanted properties.
\end{proof}

\begin{theorem}
	If $m>1$ then neither of the two adjoint orbits that compose $\Nilp$ admit Gibbs states.
\end{theorem}
\begin{proof}
	Let $\beta\in \sl(2m, \setR)$ and $\omega$ be any volume form on $\Grass(m,2m)$ (Grassmannians are orientable). We will prove that $e^{-\sprod{\beta}{\argdot}} \frac{\lambda}{\pi^*\omega}$ has an infinite integral on every fiber of $\pi$. Theorem~\ref{thmno:Fubini} will then allow to conclude.
	
	Let $E\in \Grass(m,2m)$ and $g\in \SL(2m, \setR)$ such that $g(\setR^m \times 0) = E$.
	The element $g$ relates the fiber above $E$ to the fiber above $\setR^m \times 0$:
	\[
		g^{-1}\cdot \pi^{-1}(E) = \pi^{-1}(\setR^m\times \{0\}) = F_+ \cup F_-
	.\]
	Let us define $F^E_{\pm} = g\cdot (F_\pm)$: they are the intersections of $\pi^{-1}(E)$ with each of the two adjoint orbits. We are interested in calculating the following integrals:
	\begin{equation*}
	\begin{aligned}
		\int_{F^E_{\pm}} e^{-\sprod\beta x} \frac{\lambda(x)}{\pi^*\omega(x)}
			&= \int_{F_\pm} e^{-\sprod{\beta}{\Ad_g x}} \frac{\Ad_g^*\lambda(x)}{\Ad_g^*\pi^*\omega(x)}\\
			&= \int_{F_\pm} 
				e^{-\sprod{\Ad_g^{-1}\beta}{x}} 
				\frac{\lambda(x)}{\pi^*g^*\omega(x)}
	.\end{aligned}
	\end{equation*}
	The restriction of $\sprod{\Ad_g^{-1}\beta}{\argdot}$ to the subspace of matrices of the form $\begin{pmatrix}
		0 & M\\
		0 & 0
	\end{pmatrix}$
	with $M\in \Mat(m, \setR)$ defines a linear form, therefore there exists $A_g \in \Mat(m, \setR)$ such that 
	\[
		\forall M \in \Mat(m, \setR), \,
		\sprod{\Ad_g^{-1}\beta}{\begin{pmatrix}
				0 & M\\
				0 & 0
			\end{pmatrix}}
			=
		\tr(A_gM)
	.\]
	
	From there, there are two cases:
	\begin{itemize}
	\item $A_g=0$. Then 
	\[
		\int_{F_\pm} 
			e^{-\sprod{\Ad_g^{-1}\beta}{x}} 
			\frac{\lambda(x)}{\pi^*g^*\omega(x)}
		= \int_{F_\pm} 
			1 
			\frac{\lambda(x)}{\pi^*g^*\omega(x)}
	\]
	which corresponds, according to Proposition~\ref{propno:LebesgueMeasure}, to the Lebesgue volume of $\GL^\pm(m, \setR)\subset \Mat(m, \setR)$, hence is infinite.

	\item $A_g\neq 0$. We only treat the case of $F_+$; $F_-$ can be handled similarly.
	In this case, Lemma~\ref{lmno:NegativetrAM} ensures the existence of $M_0\in \GL^+(m)$ such that $\tr(A_gM)<0$. Let us fix any norm on $\Mat(m, \setR)$; we deduce the existence of an \emph{open subset} $\opU$ of $\{ M \in \GL^+(m) \, | \, \lVert M \rVert = 1\}$ on which $M\mapsto \tr(A_g M) <0$. 
    Using a polar decomposition 
    $\Big($
    $M \mapsto \left( \lVert M \rVert, \frac{M}{\lVert M \rVert} \right)$
    $\Big)$%
    , we show that the Lebesgue integral over the open cone spanned by $\opU$ is infinite:
	\[
		\int_{\setR^*_+  \cdot \opU} e^{-\tr(A_g M)} \d M
		\stackrel{M=t U}{=}
		\int_{\opU} \int_0^\infty e^{-t \tr(A_g U)} \d t \, \sigma(U)
		= \infty
	\]
	where $\d M$ and $\d t$ are Lebesgue measures and $\sigma$ is the volume form of the sphere.
    Since $\setR^*_+ \cdot \opU \subset \GL^+(m)$, 
    this proves that 
	\[
		\int_{F_+} 	e^{-\sprod{\Ad_g^{-1}\beta}{x}} 
				\frac{\lambda(x)}{\pi^*g^*\omega(x)}
		= \infty
	.\]
	\end{itemize}
	
	We have proven that the fiber integral of $e^{-\beta} \lambda$ diverges on every fiber, which proves the result.
\end{proof}

The remaining case of $\sl(2, \setR)$ is handled in~\cite{NilpotentOrbitsGSI23}: the two nilpotent orbits have non-empty Gibbs sets, and the corresponding statistical structure is identified.
This case also appears in Neeb~\cite{NeebGibbsEnsembles} as the Lie algebra $\so(1,2)$.

%\begin{remark}
%    In fact, the argument we presented in this section can be adapted to more general index $2$ nilpotent matrices.
    %, and even for index $r+1$ nilpotent matrices with $\codim \ker x^r > 1$. 
%    This however does not cover the case $n \geqslant 3$, that is handled in Section~\ref{secno:NegationMap}.
%\end{remark}

\section{Conclusion}
We have proved that aside from the case of $\sl(2, \setR)$ handled in~\cite{NilpotentOrbitsGSI23}, the nonzero nilpotent orbits of $\sl(n, \setR)$ never admit Gibbs states.
We used three different arguments: the first one is that nilpotent orbits that are stable under negation cannot admit Gibbs states. The latter two rely on decompositions of the nilpotent orbit: first along affine lines, and then along the subset of nilpotent matrices with a common kernel.
Each of these arguments can be applied to nilpotent orbits of other Lie algebras; the former could also be applied to Hamiltonian manifolds equipped with a momentum-reversing involution.
These applications are left to future investigations.

\clearpage
\appendix

\section{Centralizer of a nilpotent matrix}
\label{secno:centralizer}

Let $n\in \setN$ and let $x\in \End(\setR^n)$ be a nilpotent endomorphism of index $r+1$.
In this section, we describe the matrices in $\End(\setR^n)$ which commute with $x$.
The only result we directly quote in the main part is Proposition~\ref{propno:DeterminantCentralizer}.

\subsection{Invariant filtrations}\label{secno:InvariantFlag}

In this section we look at invariant filtrations associated with $x$ and their consequences for the endomorphisms of $\setR^n$ that commute with $x$ of $\setR^n$.

Let $M$ be an endomorphism of $\setR^n$ that commutes with $x$. As a consequence, $M$ commutes with all powers of $x$, hence preserves their images and kernels. That is, $M$ preserves the following two filtrations:
\begin{align}
    \ker x \subset \ker x^2 \subset \dots \subset \ker x^r \subset \ker x^{r+1} = \setR^n\\
    0 = \im x^{r+1} \subset \im x^r \subset \im x^{r-1} \subset \dots \subset \im x
\end{align}

Combining these two invariant filtrations, we obtain the following property:
\begin{lemma}
    Let $M\in Z_{\End(\setR^n)} (x)$.
    Let $i, j$ be two integers such that $0\leqslant i,j \leqslant r$. 
    Then $M$ preserves $\ker x^i\cap \im x^j$. 
\end{lemma}%
In other words, $M$ preserves the following \enquote{two-indices filtration} of $\setR^n$:
\begin{equation}\label{diagno:bidrapeau}
	\begin{tikzcd}[cramped]
		\im x^{r}
			\ar[d, phantom, "\subset" sloped]
	&	
	&
	&
	&
	\\
	\ker x \cap \im x^{r-1}
			\ar[d, phantom, "\subset" sloped]
			\ar[r, phantom, "\subset"]
	&
	\im x^{r-1}
			\ar[d, phantom, "\subset" sloped]
	&
	&
	&
	\\
	\ker x \cap \im x^{r-2}
			\ar[d, phantom, "\subset" sloped]
			\ar[r, phantom, "\subset"]
	&
	\ker x^2 \cap \im x^{r-2}
			\ar[d, phantom, "\subset" sloped]
			\ar[r, phantom, "\subset"]
	&
	\im x^{r-2}
			\ar[d, phantom, "\subset" sloped]
			\ar[rd, phantom, "\ddots"]
	&
	&
	\\
	\vdots
			\ar[d, phantom, "\subset" sloped]	
			\ar[r, phantom, "\subset"]
	&
	\vdots
			\ar[d, phantom, "\subset" sloped]	
			\ar[r, phantom, "\subset"]
	&
	\vdots
			\ar[d, phantom, "\subset" sloped]	
			\ar[r, phantom, "\subset"]
	&
	\im x
			\ar[d, phantom, "\subset" sloped]
	&
	\\
	\ker x
			\ar[r, phantom, "\subset"]
	&
	\ker x^2
			\ar[r, phantom, "\subset"]
	&
		\cdots 
			\ar[r, phantom, "\subset"]
	&
	\ker x^{r}
			\ar[r, phantom, "\subset"]
	&
	\ker x^{r+1} = \setR^n
	\end{tikzcd}	
\end{equation}
This can be flattened into a single-index filtration as follows:
\begin{lemma}
    Let us define the following subspaces of $\setR^n$:
    \begin{equation}\label{eq:defFij}
        \Flag_{i,j}
            := \ker x^{i-1} + \ker x^i \cap \im x^{r+2-i-j}
        \qquad
        \forall i \in \llbracket 1 , r + 1 \rrbracket, \,
        j \in \llbracket 1 , r + 2 -i \rrbracket
    .\end{equation}
    The subspaces $\Flag_{i,j}$ are increasing according to the lexicographical order. Enumerated in the lexicographical order, they form a filtration, which takes the following form:
    \begin{equation}\label{eq:xDrapeau}
    \begin{aligned}
        &\im x^{r}
            \subset \ker x \cap \im x^{r-1}
            \subset \ker x \cap \im x^{r-2}
            \subset \dots
            \subset  \ker x \cap \im x
            \subset \ker x
            \\
        &\subset \ker x + \im x^{r-1}
        \subset \ker x + \ker x^2 \cap \im x^{r-2}
        \subset \dots
        \\
        &\dots 
        \subset \ker x^{r-1} + \im x
        \subset \ker x^{r}
        \subset \ker x^{r+1} = \setR^n
    .\end{aligned}
%    \quad.
    \end{equation}
    Then this filtration is preserved by any endomorphism $M\in Z_{\End(\setR^n)} (x)$.
\end{lemma}

In particular, in any basis adapted to the filtration~\eqref{eq:xDrapeau} (namely, a basis such that the first $\dim \Flag_{i,j}$ vectors belong to $\Flag_{i,j}$), every element of $Z_{\End(\setR^n)} (x)$ has a block-wise triangular matrix. 

In order to describe the centralizer of $x$ precisely, we introduce a splitting (a compatible gradation) of the filtration $\Flag_{i,j}$ in the next section.

\subsection{$x$-chains}

We introduce here a decomposition of $\setR^n$ that will be essential to describe the centralizer of $x$.

Consider a basis $\Base$ of $\setR^n$ in which $x$ takes a Jordan canonical form. Each Jordan block corresponds to the action of $x$ on a vector subspace that is spanned by a \emph{cyclic} vector under the action of $x$. For every $j\leqslant r+1$, let $\mathcal{T}_j$ be the set of those cyclic vectors of $\Base$ that generate a subspace of dimension $j$ under the action of $x$.
The families $\mathcal{T}_j$ satisfy the following two properties:
\begin{enumerate}
    \item $0\notin x^{j-1}(\mathcal{T}_j)$,
    \item $\{
                x^i v
            \}_{
            \substack{
                1 \leqslant j \leqslant r+1
                \\
                v \in \mathcal{T}_i
                \\
                0\leqslant i < j}
                }$
            forms a basis of $\setR^n$.
\end{enumerate}
%
%For every $0< j \leqslant r$, let $T_j \subsetneq \ker x^j$ be a supplementary subspace to $\ker x^{j-1} + \im x \cap \ker x^j$:
%\[
%    \ker x^{j} = \left(
%        \ker x^{j-1} + \im x \cap \ker x^j \right)
%        \oplus  T_j
%\]
%
For every $0< j \leqslant r+1$, let $E_{0,j} = \Vect(\mathcal{T}_j) \subset \ker x^{j}$.
Then $x^{j-1}$ is injective on $E_{0,j}$ and $x^j$ vanishes on $E_{0,j}$. We define
\begin{equation}\label{eqno:DefEij}
    E_{i,j} := x^i (E_{0,j})
    \qquad
    i<j
.\end{equation}
%In particular, $E_{0,j} = T_j$.
The definition is extended to $E_{i,j} = \{0\}$ when $i\geqslant j$ (note that $E_{i,j}$ can be $0$ even when $i<j$, when $\mathcal{T}_j = \emptyset$).
The subspaces $\left( E_{i,j} \right)_{i<j\leqslant r+1}$ are in direct sum and form a decomposition of $\setR^n$. Let us write $p_{i,j}$ for the associated projectors on each $E_{i,j}$.

We illustrate the subspaces $E_{i,j}$ with the action of $x$ in Figure~\ref{fig:Eij}, for $r+1=4$.

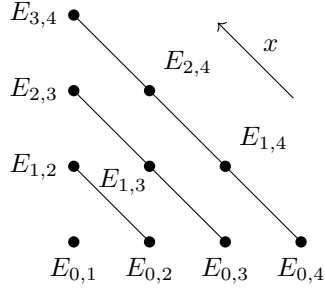
\begin{figure}
    \centering
    \caption{Decomposition of $\setR^n$ into $(E_{i,j})_{i<j\leqslant 4}$.}
    \label{fig:Eij}
    \begin{tikzpicture}
    	[dot/.style={circle, fill, inner sep=1.5pt}]
    	\node at (0,0) [dot, label=below:$E_{0,1}$] {};
    	\node at (1,0) [dot, label=below:$E_{0,2}$] {};
        \node at (2,0) [dot, label=below:$E_{0,3}$] {};
    	\node at (3,0) [dot, label=below:$E_{0,4}$] {};
    	\node at (0,1) [dot, label=left:$E_{1,2}$] {};
    	\node at (1,1) [dot, label={[label distance = -6pt]225:$E_{1,3}$}] {};
    	\node at (2,1) [dot, label=above right:$E_{1,4}$] {};
    	\node at (0,2) [dot, label=left:$E_{2,3}$] {};
    	\node at (1,2) [dot, label=above right:$E_{2,4}$] {};
    	\node at (0,3) [dot, label=left:$E_{3,4}$] {};
	
    	\draw (1,0) -- (0,1);
    	\draw (2,0) -- (0,2);
    	\draw (3,0) -- (0,3);
	
        \draw[->] (2.9, 1.9) -- (1.9, 2.9);
        \node at (2.6,2.6) {$x$};
\end{tikzpicture}
\end{figure}

\begin{lemma}
    Let $i, j \in \llbracket 0, r +1 \rrbracket$ be such that $i+1 < j$.
    Then $x \circ p_{i,j} = p_{i+1,j} \circ x$ (with implicit embeddings $E_{i,j} \hookrightarrow \setR^n$).
\end{lemma}

\begin{proof}
    Let $v\in\setR^n$ and $v_{k,l} = p_{k,l} v \in E_{k,l}$ for $0\leqslant k < l \leqslant r +1$, so that $v = \sum_{k<l} v_{k,l}$.
    Then $x v_{k,l} \in E_{k+1, l}$.
    Consequently, when $i+1 <j$,
\[
    p_{i+1, j}(xv)
        = p_{i+1, j} \sum_{k< l} xv_{k,l} 
        = x v_{i,j}
.\]
\end{proof}
The subspaces $E_{i,j}$ can be used to define a splitting of the filtration~\eqref{eq:xDrapeau}.
\begin{lemma}\label{lmno:KerImEij}
    Let $i,j \leqslant r$ be two integers.
    \begin{enumerate}
        \item $\ker x^j
%            = \bigoplus_{
%            \substack{ i',j' \leqslant r \text{ such that} \\ 
%                j'-j \leqslant i' < j'}}
%                E_{i',j'}$
            = \bigoplus_{j'=1}^{r+1}
            \bigoplus_{i'=\max(0, j'-j)}^{j'-1}
                E_{i',j'}$
            ,
%        \item $\im x^i
%            = \bigoplus_{
%            \substack{ i \leqslant i' < j'}}
%                E_{i',j'}$
        \item $\im x^i
            = \bigoplus_{j'=i+1}^{r+1}
            \bigoplus_{i'=1}^{j'-1}
                E_{i',j'}$
            ,
%        \item $\ker x^j \cap \im x^i
%            = \bigoplus_{
%            \substack{ \max(i,j'-j) \leqslant i' < j'}}
%                E_{i',j'}$
        \item $\ker x^j \cap \im x^i
            = \bigoplus_{j'=i+1}^{r+1}
            \bigoplus_{i=\max(i, j'-j)}^{j'-1}
                E_{i',j'}$
            .
    \end{enumerate}
\end{lemma}
\begin{proof}
    Let us notice first that for all integers $i'<j'$ the restriction of $x^{i'}$ induces a linear isomorphism $E_{0,j'} \isom E_{i',j'}$. In particular if $E_{0,j'}\neq 0$ then for every $i'<j'$, the subspace $E_{i',j'}$ is non-trivial.
    Furthermore, given two integers $i'<i''<j'$, the restriction of $x^{i''-i'}$ defines a linear isomorphism $E_{i',j'}\isom E_{i'', j'}$.
    \begin{enumerate}
        \item Let $i', j' \leqslant r+1$ be two integers with $i'<j'$.
        If $j+i' \geqslant j'$, then $x^j(E_{i',j'}) = x^{j+i'} (E_{0,j'})=0$. Therefore $E_{i',j'}\subset \ker x^j$.

        On the other hand, if $j+i' < j'$ then the restriction of $x^{j}$ defines an isomorphism $E_{i',j'}\isom E_{j+i', j'}$. Consequently, $x^j$ is injective on 
%        $\bigoplus_{i'<j'-j} E_{i',j'}$.
        $\bigoplus_{j'=1}^{r+1} \bigoplus_{i'=0}^{j'-j-1} E_{i',j'}$.

        In conclusion, 
%        $\ker x^j = \bigoplus_{j'-j \leqslant i' < j'} E_{i',j'}$.
        $\ker x^j = 
            \bigoplus_{j'=1}^{r+1}
            \bigoplus_{i'=\max(0, j'-j)}^{j'-1}
                E_{i',j'}$.
        \item Applying $x^i$ to the decomposition 
%        $\setR^n = \bigoplus_{i' < j' \leqslant r+1} E_{i',j'}$, we obtain
        $\setR^n = \bigoplus_{j'=1}^{r+1} \bigoplus_{i'=0}^{j'-1} E_{i',j'}$, we obtain
        \[
            \im x^i = 
%            \sum_{i'<j'\leqslant r+1}
            \sum_{j'=1}^{r+1} \sum_{i'=0}^{j'-1}
            x^i(E_{i',j'})
%            = \sum_{i'<j'-i\leqslant r+1}
            = \sum_{j'=1}^{r+1} \sum_{i'=0}^{j'-i}
            E_{i+i', j'}
%            = \sum_{i\leqslant i' < j' \leqslant r+1}
            = \sum_{j'=i+1}^{r+1} \sum_{i'=i}^{j'-1}
            E_{i',j'}
        .\]
        This gives the result since the $E_{i',j'}$ are in direct sum.

        \item This is a direct consequence of 1. and 2.
    \end{enumerate}
\end{proof}

We illustrate the filtrations $\ker x^i$ and $\im x^j$ in Figure~\ref{fig:EijKerIm}, for $r+1=4$. The sum of the $E_{i,j}$ contained in any upper left quadrant (delimited by a lower right corner) is stable under the action of $Z_{\End(\setR^n)}(x)$.

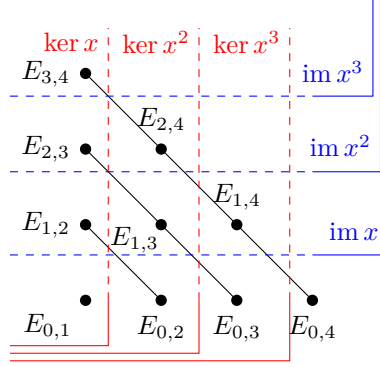
\begin{figure}
    \centering
    \caption{Kernel and images of powers of $x$ decomposed into $(E_{i,j})_{i<j\leqslant 4}$.}
    \label{fig:EijKerIm}
    \begin{tikzpicture}
    	[dot/.style={circle, fill, inner sep=1.5pt}]
    	\node at (0,0) [dot, label=below left:$E_{0,1}$] {};
    	\node at (1,0) [dot, label=below:$E_{0,2}$] {};
        \node at (2,0) [dot, label=below:$E_{0,3}$] {};
    	\node at (3,0) [dot, label=below:$E_{0,4}$] {};
    	\node at (0,1) [dot, label=left:$E_{1,2}$] {};
    	\node at (1,1) [dot, label={[label distance = -6pt]225:$E_{1,3}$}] {};
    	\node at (2,1) [dot, label=above:$E_{1,4}$] {};
    	\node at (0,2) [dot, label=left:$E_{2,3}$] {};
    	\node at (1,2) [dot, label=above:$E_{2,4}$] {};
    	\node at (0,3) [dot, label=left:$E_{3,4}$] {};
	
    	\draw (1,0) -- (0,1);
    	\draw (2,0) -- (0,2);
    	\draw (3,0) -- (0,3);

    	\draw (-1,-0.6) -- (0.3,-0.6) -- (0.3, 0) [color = red];
    	\draw (0.3, 0) -- (0.3, 3.6) [color = red, dashed]
			node[pos=0.95, left] {$\ker x$};
	
    	\draw (-1,-0.7) -- (1.5,-0.7) -- (1.5, 0) [color = red];
    	\draw (1.5, 0) -- (1.5, 3.6) [color = red, dashed]
			node[pos=0.95, left] {$\ker x^2$};
		
    	\draw (-1,-0.8) -- (2.7,-0.8) -- (2.7, 0) [color = red];
    	\draw (2.7, 0) -- (2.7, 3.6) [color = red, dashed]
			node[pos=0.95, left] {$\ker x^3$};

        \draw (-1,2.7) -- (3,2.7) [color = blue, dashed];
    	\draw (3,2.7) -- (3.8,2.7) -- (3.8,4) [color=blue]
			node[near start, left] {$\im x^3$};
	
    	\draw (-1,1.7) -- (3,1.7) [color = blue, dashed];
    	\draw (3,1.7) -- (3.9,1.7) -- (3.9,4) [color=blue]
			node[pos=0.15, left] {$\im x^2$};
	
    	\draw (-1,0.6) -- (3,0.6) [color = blue, dashed];
    	\draw (3,0.6) -- (4,0.6) -- (4,4) [color=blue]
			node[pos=0.09, left] {$\im x$};
\end{tikzpicture}
\end{figure}

According to Lemma~\ref{lmno:KerImEij}, the spaces $E_{i,j}$ define a splitting of the filtration~\eqref{eq:xDrapeau} in the following sense: the filtration takes the form
\begin{equation}\label{eq:FijEij}
    \Flag_{i,j}
    = \ker x^{i-1} + \ker x^{i}\cap \im x^{r+2-i-j}
%    = \bigoplus_{
%        \substack{ i',j' \leqslant r+1 \text{ such that} \\ 
%            j'+1-i \leqslant i' < j'}}
%            E_{i',j'}
    = \bigoplus_{j'=1}^{r+1} 
        \bigoplus_{i'=j'+1-i}^{j'-1}
        \oplus
      \bigoplus_{j'=r+2-j}^{r+1}
            E_{j'-i,j'}
.\end{equation}

\subsection{Characterization of the centralizer of $x$}\label{secno:centralizerx}

We can now characterize the endomorphisms of $\setR^n$ that commute with $x$.

\begin{lemma}\label{lm:DiagrammeCommutatifM}
    Let $M\in \End(\setR^n)$.
    Then $M$ commutes with $x$ if and only if 
    for all integers $i,j$ such that $0 \leqslant i<j \leqslant r+1$ , $M(E_{i,j}) \subset \ker x^{j-i}$
	and the following diagram is commutative:
	\begin{equation}\label{diag:Mx}
	\begin{tikzcd}
        E_{0,j}
			\ar[d, "x^i|_{E_{0,j}}"', "\sim" sloped]
			\ar[r, "M|_{E_{0,j}}"]
			&
		\ker x^j
			\ar[d, "x^{i}|_{\ker x^{j}}" outer sep=5pt]
		\\
		E_{i,j}
			\ar[r, "M|_{E_{i,j}}"']
			&
		\ker x^{j-i}
	\end{tikzcd}
    \quad.
	\end{equation}
\end{lemma}

\begin{proof}
    If $M$ commutes with $x$, it preserves $\ker x^{j-i}$ for $i<j$. Since $E_{i,j} \subset \ker x^{j-i}$, we conclude that $M(E_{i,j}) \subset M(\ker x^{j-i}) \subset \ker x^{j-i}$.

    Now, assume that $M$ maps each $E_{i,j}$ to $\ker x^{j-i}$ when $i<j$.
    The condition $xM=Mx$ is satisfied on $\setR^n$ if and only if it holds on each subspace $E_{i,j}$. When $i+1=j$, both sides are $0$ and the condition is trivial. When $i+1<j$, it amounts to the following diagram being commutative:
	\begin{equation}\label{diag:MEij}
%	    \begin{tikzcd}
%			\ker x^{j-i-1}
%				&
%			\ker x^{j-i}
%				\ar[l, "x|_{\ker x^{j-i}}"' outer sep=5pt]
%			\\
%			E_{i+1,j}
%				\ar[u, "M|_{E_{i+1,j}}"]
%				&
%			E_{i,j}
%				\ar[l, "x|_{E_{i,j}}", "\sim"']
%				\ar[u, "M|_{E_{i,j}}"']
%	    \end{tikzcd}
    \begin{tikzcd}
        E_{i,j}
			\ar[d, "x|_{E_{i,j}}"', "\sim" sloped]
			\ar[r, "M|_{E_{i,j}}"]
			&
		\ker x^{j-i}
			\ar[d, "x^{i}|_{\ker x^{j-i}}" outer sep=5pt]
		\\
		E_{i+1,j}
			\ar[r, "M|_{E_{i+1,j}}"']
			&
		\ker x^{j-i-1}
	\end{tikzcd}
    \quad.    
    \end{equation}
    By concatenating the Diagrams~\eqref{diag:MEij} with varying $i$, we deduce that the commutativeness of the set of the diagrams~\eqref{diag:MEij} with $0 \leqslant i < j \leqslant r+1$ is equivalent to the commutativeness of the set of the diagrams~\eqref{diag:Mx} with $1 \leqslant i + 1 < j \leqslant r+1$.
\end{proof}

\begin{remark}
    The commutativeness of Diagram~\eqref{diag:Mx} implies that $M(E_{i,j}) \subset \ker x^{j-i} \cap \im x^i$. This fact will be developed in Section~\ref{secno:InvariantFlag} to triangularize $M$.
\end{remark}

\begin{proposition}[Centralizer of $x$ in $\End(\setR^n)$]\label{prop:Centralizerx}
    Let for all $1\leqslant k \leqslant r$ be a linear map $a_k : E_{0,k} \to \ker x^k$.
    There exists a unique endomorphism $A\in \End(\setR^n)$ that commutes with $x$ such that $A|_{E_{0,k}} = a_k$.
\end{proposition}

\begin{proof}
    We start with uniqueness: it is a direct consequence of Lemma~\ref{lm:DiagrammeCommutatifM}.

    We now give the construction of $A$ given the family $(a_k)_{1\leqslant k \leqslant r+1}$. 
    Let $k\in \llbracket 1, r+1 \rrbracket$.
    The commutation condition of Diagram~\eqref{diag:MEij} takes the form:
	\[
		(A x^i)|_{E_{0,k}}
		= A|_{E_{i,k}} x^i |_{E_{0,k}} 
		= x^i A|_{E_{0,k}} 
		= x^i a_k
	.\] 
    Since $x^i$ establishes a bijection between $E_{0,k}$ and $E_{i,k}$, there is a unique solution $A|_{E_{i,k}}$.
    The direct sum of these maps $(A|_{E_{i,k}})_{i}$ defines on $\bigoplus_i x^i E_{0,k} = \bigoplus_i E_{i,k}$ a map to $\setR^n$ that commutes with $x$, since the condition of Lemma~\ref{lm:DiagrammeCommutatifM} is satisfied.

    Taking the direct sum of these maps for $1\leqslant k \leqslant r+1$, we obtain a map $A : \bigoplus_{i,j} E_{i,j} = \setR^n \to  \setR^n$. The map $A$ commutes with $x$ and its restriction to $E_{0,k}$ coincides with $a_k$ for all $1\leqslant k \leqslant r+1$.

\end{proof}

Let us give a more explicit decomposition of a generic element $M$ of $Z_{\End(\setR^n)}(x)$, using the subspaces $E_{i,j}$. According to Lemma~\ref{lmno:KerImEij}, 
$\ker x^j
%     = \bigoplus_{
%        \substack{ i',j' \leqslant r +1\text{ such that} \\ 
%            j'-j \leqslant i' < j'}}
%            E_{i',j'}$.
            = \bigoplus_{j'=1}^{r+1}
            \bigoplus_{i'=\max(0, j'-j)}^{j'-1}
                E_{i',j'}$.

For $k\in \llbracket 1, r+1 \rrbracket$, the application $a_k$ defined in Proposition~\ref{prop:Centralizerx} can be decomposed as a direct sum of maps
\[
    M^k_{i,j} : E_{0,k} \to E_ {i,j},
    \qquad
    \text{ with }
    k+i \geqslant j
.\]

Now, the action of $M$ on a given vector $v\in E_{i,j}$, can be expressed using the components $M^k_{i,j}$, according to Diagram~\eqref{diag:Mx}:
\[\begin{aligned}
    M v
    &= Mx^i \left( x^i|_{E_{0,j}} \right)^{-1} v\\
    &= x^i M \left( x^i|_{E_{0,j}} \right)^{-1} v\\
%    &= x^i \sum_{i'+j \geqslant j'} M^j_{i', j'} \left( x^i|_{E_{0,j}} \right)^{-1} v\\
    &= x^i \sum_{j'=1}^{r+1} \quad \sum_{i'=\max(0, j'-j)}^{j'-1}
           M^j_{i',j'} \left( x^i|_{E_{0,j}} \right)^{-1} v\\
%    &= \sum_{j'-i' \leqslant j}  x^i M^j_{i', j'} \left( x^i|_{E_{0,j}} \right)^{-1} v\\
    &= \sum_{j'=1}^{r+1} \quad \sum_{i'=\max(0, j'-j)}^{j'-1}
            x^i M^j_{i', j'} \left( x^i|_{E_{0,j}} \right)^{-1} v\\
%    &= \sum_{i < j'-i' \leqslant j}  x^i M^j_{i', j'} \left( x^i|_{E_{0,j}} \right)^{-1} v
    &= \sum_{j'=1}^{r+1} \quad \sum_{i'=\max(0, j'-j)}^{j'-i - 1}
    x^i M^j_{i', j'} \left( x^i|_{E_{0,j}} \right)^{-1} v
.\end{aligned}
\]
The last equality holds since $i+i' \geqslant j'$ implies that $x^i (E_{i',j'}) = 0$. The computation can be represented with the following commutative diagram:
\begin{equation}\label{diag:Mkij}
%	\begin{tikzcd}
%        E_{i'+i,j'}
%            &
%        E_{i',j'}
%			\ar[l, "x^{i}|_{E_{i',j'}}"' outer sep=5pt, "\sim"]
%        \\
%		\ker x^{j-i}
%            \ar[u, "p_{i'+i,j'}"]
%			&
%		\ker x^j
%            \ar[u, "p_{i',j'}"']
%			\ar[l, "x^{i}|_{\ker x^{j}}"' outer sep=5pt]
%		\\
%		E_{i,j}
%			\ar[u, "M|_{E_{i,j}}"]
%			&
%		E_{0,j}
%			\ar[l, "x^i|_{E_{0,j}}", "\sim"']
%			\ar[u, "M|_{E_{0,j}}"']
%            \ar[uu, "M^j_{i',j'}"', bend right=75]
%\end{tikzcd}\quad.\end{equation}
%
%\begin{equation}\label{diag:Mkij-horizontal}
\begin{tikzcd}[sep = large]
    E_{0,j}
        \ar[d, "x^i|_{E_{0,j}}"', "\sim" sloped]
        \ar[r, "M|_{E_{0,j}}"']
        \ar[rr, "M^j_{i',j'}", bend left=25]
&
    \ker x^j
        \ar[d, "x^{i}|_{\ker x^{j}}" outer sep=5pt]
        \ar[r, "p_{i',j'}"']
&
    E_{i',j'}
        \ar[d, "x^{i}|_{E_{i',j'}}" outer sep=5pt, "\sim"' sloped]
\\
    E_{i,j}
        \ar[r, "M|_{E_{i,j}}"']
&
    \ker x^{j-i}
        \ar[r, "p_{i'+i,j'}"']
&
    E_{i'+i,j'}
\end{tikzcd}\quad.
\end{equation}

Consequently, all components $E_{i, j}\to E_{i',j'}$ (with $i<j, i'<j'$) can be identified to components $M^k_{i,j}$ by using $x^i$ to identify $E_{i,j}$ with $E_{0,j}$. Here is a concrete example that uses a different naming convention from $M^k_{i,j}$ for readability purposes.

\begin{example}
    Suppose $r+1=3$ and that $x$ has only Jordan blocks of size $2$ and $3$, so that $\setR^n = E_{0,2} \oplus E_{1,2} \oplus E_{0,3} \oplus E_{1,3} \oplus E_{2,3}$. We order the subspaces as follows:
    \[
        \setR^n
            = \overbrace{
                  \underbrace{
                        E_{2,3} \oplus E_{1,2}
                  }_{\ker x}
                  \oplus E_{1,3} \oplus E_{0,2}
                }^{\ker x^2}
            \oplus E_{0,3}
    .\]
    Let $\Base_3$ (resp. $\Base_2$) be a basis of $E_{0,3}$ (resp. $E_{0,2}$). Then $x(\Base_3)$ (resp. $x(\Base_2)$) is a basis of $E_{1,3}$ (resp. $E_{1,2}$) and $x^2(\Base_3)$ is a basis of $E_{2,3}$. We obtain by concatenation the following basis of $\setR^n$:
    \[
        \Base
        = \big(
            x^2(\Base_3),
            x(\Base_2),
            x(\Base_3),
            \Base_2,
            \Base_3
        \big)
    .\]
    
    In the basis $\Base$, the centralizer of $x$ is the set of matrices of the following form, with labels to indicate the decomposition of $\setR^n$ into subspaces (empty entries are zero):

    \NiceMatrixOptions{
    code-for-first-row = {\RowStyle[cell-space-bottom-limit=7pt]{}}}% Met de l'espace entre la matrice et les labels de chaque ligne/colonne
    
    \[
        \begin{pNiceMatrix}[first-row, first-col]
            & E_{2,3}
            & E_{1,2}
            & E_{1,3}
            & E_{0,2} 
            & E_{0,3}
           \\
            E_{2,3} \quad
            & A_1
            & B_2
            & C_1
            & D_2
            & E_1
            \\
            E_{1,2} \quad
            &
            & A_2
            & B_1
            & C_2
            & D_1
            \\
            E_{1,3} \quad
            &
            & 
            & A_1
            & B_2
            & C_1
            \\
            E_{0,2} \quad
            &
            & 
            & 
            & A_2
            & B_1
            \\
            E_{0,3} \quad
            &
            & 
            & 
            & 
            & A_1           
        \end{pNiceMatrix}
    \quad.\]
\end{example}

%\begin{proposition}
%    Let $M\in Z_{\End(\setR^n)} (x)$.
%    Then $M$ is an invertible endomorphism if and only for all $0 \leqslant j \leqslant r$, the endomorphism $p_{0,j}\circ M|_{T_j}$ is invertible.
%\end{proposition}
%
%\begin{proof}
%    $M$ preserves the flag $\Flag_{i,j}$.
%    According to Equation~\eqref{eq:FijEij}, 
%    \begin{align}\label{eqn: flag_complements}
%        \Flag_{i, j+1}
%            &= 
%        \Flag_{i,j} \oplus E_{r-j+1-i,r-j+1},
%        \qquad j\leqslant r-i,
%        \\
%        \Flag_{i+1, 1}
%            &= \Flag_{i, r+1-i} \oplus E_{r-i-1, r}
%        .
%    \end{align} 
   % Notice that every $E_{i,j}$ for $1 \leqslant i, j \leqslant r$ such that $i<j$ appears as a \enquote{difference} of two subsequent subspaces of the flag $\Flag$.
  %  Therefore, $M$ is invertible if and only if for every $1 \leqslant i, j \leqslant r$ such that $i<j$, $p_{i,j}\circ M|_{E_{i,j}}$ is an invertible endomorphism of $E_{i,j}$.
 %   According to Diagram~\eqref{diag:Mkij}, this is equivalent to each $M^j_{0,j} = p_{0,j}\circ M|_{E_{0,j}}$ being invertible.
%\end{proof}

The determinant of an endomorphism that commutes with $x$ is easily obtained from its components:
\begin{proposition}\label{propno:DeterminantCentralizer}
    Let $M\in Z_{\End(\setR^n)} (x)$.
    Define for all $0 \leqslant j \leqslant r+1$
    \[
        M_j:= p_{0,j}\circ M|_{E_{0,j}} \in \End(E_{0,j})
    .\]

    Then 
    \[
        \det(M) = \prod_{j=1}^n \det(M_j)^j
    .\]
\end{proposition}

\begin{proof}
    $M$ preserves the filtration $\Flag_{i,j}$.
    According to Equation~\eqref{eq:FijEij}, 
    \begin{align}\label{eqn: flag_complements}
        \Flag_{i, j+1}
            &= 
        \Flag_{i,j} \oplus E_{r-j+1-i,r-j+1},
        \qquad j\leqslant r+1-i,
        \\
        \Flag_{i+1, 1}
            &= \Flag_{i, r+1-i} \oplus E_{r-i-1, r+1}
        .
    \end{align} 
    Notice that for every $1 \leqslant i, j \leqslant r+1$ such that $i<j$, the space $E_{i,j}$ appears as a \enquote{difference} (quotient) of two subsequent subspaces of the filtration $\Flag$.
    Therefore, the determinant of $M$ can be expressed as a block-wise determinant:
    \[
        \det(M) = \prod_{0 \leqslant  i<j \leqslant r+1} \det(p_{i,j}\circ M|_{E_{i,j}}).
    \]
    According to Diagram~\eqref{diag:Mkij}, $\det(p_{i,j}\circ M|_{E_{i,j}}) = \det(p_{0,j}\circ M|_{E_{0,j}})$, which proves the result.
\end{proof}
As a consequence, we obtain an invertibility criterion for the endomorphisms commuting with $x$:
\begin{corollary}
    Let $M\in Z_{\End(\setR^n)} (x)$.
    Then $M$ is an invertible endomorphism if and only for all $0 \leqslant j \leqslant r+1$, the endomorphism $p_{0,j}\circ M|_{E_{0,j}}$ is invertible.
\end{corollary}

\printbibliography

@book{SSDEng,
  title={Structure of Dynamical Systems: A Symplectic View of Physics},
  author={Souriau, J.M.},
  isbn={9780817636951},
  lccn={97000300},
  series={Progress in Mathematics},
  year={1997},
  publisher={Birkh{\"a}user Boston}
}

@inproceedings{SouriauMecaStat,
    author = {Souriau, Jean-Marie},
    title = {M\'ecanique Statistique, groupes de Lie et cosmologie},
    booktitle = {G\'eom\'etrie symplectique et physique math\'ematique},
    series = {Colloques Internationaux du CNRS},
    year = 1974,
    language = {french}
}

@article{MarleExamples,
	author = {Marle, Charles-Michel},
	title = {{Examples of Gibbs States of Mechanical Systems with Symmetries}},
	volume = {58},
	journal = {Journal of Geometry and Symmetry in Physics},
	publisher = {Bulgarian Academy of Sciences, Institute of Mechanics},
	pages = {55 -- 79},
	year = {2020},
	doi = {10.7546/jgsp-58-2020-55-79},
	URL = {https://doi.org/10.7546/jgsp-58-2020-55-79}
}

@article{MarleGibbsStates,
  title = {On Gibbs States of Mechanical Systems with Symmetries},
  volume = {57},
  ISSN = {1314-5673},
  url = {http://dx.doi.org/10.7546/jgsp-57-2020-45-85},
  DOI = {10.7546/jgsp-57-2020-45-85},
  journal = {Journal of Geometry and Symmetry in Physics},
  publisher = {Prof. Marin Drinov Publishing House of BAS (Bulgarian Academy of Sciences)},
  author = {Marle,  Charles-Michel},
  year = {2020},
  pages = {45–85}
}

@misc{HessianGas,
      title={The Hessian geometry of the ideal gas in rotation}, 
      author={Jérémie Pierard de Maujouy},
      year={2024},
      eprint={2404.09035},
      archivePrefix={arXiv},
      primaryClass={math-ph},
      url={https://arxiv.org/abs/2404.09035}, 
}

@article{TojoYoshino,
  title = {A method to construct exponential families by representation theory},
  volume = {5},
  ISSN = {2511-249X},
  url = {http://dx.doi.org/10.1007/s41884-022-00072-y},
  DOI = {10.1007/s41884-022-00072-y},
  number = {2},
  journal = {Information Geometry},
  publisher = {Springer Science and Business Media LLC},
  author = {Tojo,  Koichi and Yoshino,  Taro},
  year = {2022},
  month = oct,
  pages = {493–510}
}

@book{KirillovOrbitMethod,
  title={Lectures on the Orbit Method},
  author={Kirillov, A.A.},
  isbn={9780821835302},
  lccn={2004047940},
  series={Graduate studies in mathematics},
  year={2004},
  publisher={American Mathematical Society}
}

@misc{CompactCoadjointOrbits,
      title={Compact Coadjoint Orbits}, 
      author={John Rawnsley},
      year={2003},
      eprint={math/0306336},
      archivePrefix={arXiv},
      primaryClass={math.RT},
      url={https://arxiv.org/abs/math/0306336}, 
}

@InProceedings{NilpotentOrbitsGSI23,
	author={Bieliavsky, Pierre
	and Dendoncker, Valentin
	and Neuttiens, Guillaume
	and de Maujouy, J{\'e}r{\'e}mie Pierard},
	editor={Nielsen, Frank
	and Barbaresco, Fr{\'e}d{\'e}ric},
	title={Riemannian Geometry of Gibbs Cones Associated to Nilpotent Orbits of Simple Lie Groups},
	booktitle={Geometric Science of Information},
	year={2023},
	publisher={Springer Nature Switzerland},
	pages={144--151},
	isbn={978-3-031-38299-4}
}

@book{MichorTopicsDG,
  title = {Topics in Differential Geometry},
  ISBN = {9781470411619},
  ISSN = {1065-7339},
  url = {http://dx.doi.org/10.1090/GSM/093},
  DOI = {10.1090/gsm/093},
  journal = {Graduate Studies in Mathematics},
  publisher = {American Mathematical Society},
  author = {Michor,  Peter},
  year = {2008},
  month = jul 
}

@book{LeeManifolds,
  title={Introduction to Smooth Manifolds},
  edition ={2},
  author={Lee, J.M.},
  isbn={9780387954486},
  lccn={2002070454},
  series={Graduate Texts in Mathematics},
  year={2003},
  publisher={Springer}
}

@misc{NeebGibbsEnsembles,
      title={A classification of coadjoint orbits carrying Gibbs ensembles}, 
      author={Karl-Hermann Neeb},
      year={2026},
      eprint={2601.04934},
      archivePrefix={arXiv},
      primaryClass={math.SG},
      url={https://arxiv.org/abs/2601.04934}, 
}

@book{SemiClassicalAnalysis,
    author = {Guillemin, Victor and Sternberg, Shlomo},
    title = {Semi-Classical Analysis},
    publisher = {International Press of Boston Inc.},
    year = {2013} 
}

\end{document}